\documentclass[10pt]{amsart}
\usepackage{amscd}

\usepackage{amssymb}

\newcommand{\bT}{{\mathbb T}} 
\newcommand{\bP}{{\mathbb P}}

\newcommand{\bZ}{{\mathbb Z}}
\newcommand{\bR}{{\mathbb R}}

\newcommand{\bC}{{\mathbb C}}

\newcommand{\bQ}{{\mathbb Q}}

\newcommand{\bG}{{\mathbb G}}

\newtheorem{thm}{Theorem}[section]
\newtheorem{lemma}[thm]{Lemma}
\newtheorem{cor}[thm]{Corollary}

\numberwithin{equation}{section}

\begin{document}

\title{Etale cohomologies of quadrics over $\bR$}
 
\author{Nobuaki Yagita}

\address{ faculty of Education, 
Ibaraki University,
Mito, Ibaraki, Japan}
 
\email{ nobuaki.yagita.math@vc.ibaraki.ac.jp, }


\begin{abstract}
In this paper, we study etale cohomologies of quadrics over $\bR$.
An element in the etale cohomology is  called algebraic, if it is in the image of the cycle map
from the Chow ring. 
In this paper, we compute the etale cohomology of norm quadrics, and  give examples which have
 many non-algebraic elements. 
\end{abstract}

\maketitle

\section{Introduction}
Let $X$ be a smooth algebraic variety over a field $k$.
We consider the cycle map
 from the Chow ring to the etale cohomoly (of $2$-adic integer  $\bZ_2$ coefficients)
\[cl : CH^*(X)\to H^{2*}_{et}(X;\bZ_2(*)).\]
We say that $x\in H^{2*}(X;\bZ_2(*))$ is algebraic if $x\in Im(cl)$.  (See 
$\S 3$ below  for  the definition of $H_{et}^{2*}(X;\bZ_2(*))$ 
in this paper.)

The above cycle maps and algebraic properties
are studied for examples in  [Be],[Sc-Su].
In particular Benoist shows many examples of existence of
non-algebraic elements in various situations. 
We give another examples of quadrics over $\bR$ with many non-algebraic elements,
which seem to relate some question in 4.5 in [Be].

\begin{thm} Let   $Q^d$ be the anisotropic quadric of dimension  $d$ over $\bR$.  
Then the  cohomology 
$H_{et}^{2*}(Q^d;Z_2(*))$ 
has  a non-algebraic element
if and only if $d\ge 2^3-1=7$.

\end{thm}

\begin{thm} Let   $Q^d$ be the anisotropic quadric of dimension  $d=2^n-1$ (i.e. the norm variety).  
Then, for each $c =0\ mod(4)$ with $4\le c \le 2^{n+1}-12$, the cohomology 
$H_{et}^{c}(Q^d;Z_2)$ 
has a non-algebraic element.
However,  the cohomology $H_{et}^{2*}(Q^d;Z_2(*))/(2-torsion)$
does not have a non-algebraic element. 
\end{thm}

The above theorems follow from
\begin{thm}
Let $M_n\subset Q^{2^n-1}$ be the Rost motive [Ro] 
of the 
norm variety.  Then there are elements 
$\pi\in H^{2^{n+1}-2}_{et}(M_n; \bZ_2(1))$ and 
$\bar \rho_{4m}\in
H^{4m}_{et}(M_n;\bZ_2)$ such that 
\[ H^{2*}_{et}(M_n;\bZ_2(*))\cong
\bZ_2\{1,\pi\}\oplus \bZ/2\{\bar \rho_4,
\bar \rho_8,...,\bar \rho_{2^{n+1}-4} \}\]
\[ \cong \bZ_2\{1,\pi\}\oplus \bZ/2[\bar \rho_4]^+/(
\bar \rho_4^{2^{n-1}})
\]
where  $A\{a,b,...\}$ means the  $A$-free module generated by $a,b,...$
The image of the cycle map is given \ \
\[CH^*(M_n)\otimes \bZ_2\cong
\bZ_2\{1,\pi\}\oplus \bZ/2\{\bar \rho_{2^{n+1}-2^{n}},
\bar \rho_{2^{n+1}-2^{n-1}},...,\bar \rho_{2^{n+1}-4} \} .\]
\end{thm}

\section{Rost motive over $\bR$ with coefficients $\bZ/2$}

Let $X$ be a (motive of) a smooth variety over the field
 $\bR$ of real numbers,
and we consider the cohomologies of $\bZ/2$ coefficients.
In this paper the $mod(2)$ etale cohomology means the 
motivic cohomology of the same first and the second degrees
$ H_{et}^*(X;\bZ/2)\cong 
H^{*,*}(X;\bZ/2).$  
The cohomology operation of the etale theory means that of the motivic theory
identifying $H_{et}^*(X;\bZ/2)\cong lim_{N\to \infty}\tau^N
H^{*,*+N}(X;\bZ/2)$.

It is well known ([Vo1], [Vo2])
\[ H_{et}^*(Spec(\bC);\bZ/2)\cong \bZ/2, \quad 
 H^{*,*'}(Spec(\bC);\bZ/2)\cong \bZ/2[\tau], \]
\[ H_{et}^*(Spec(\bR);\bZ/2)\cong \bZ/2[\rho], \quad 
 H^{*,*'}(Spec(\bR);\bZ/2)\cong \bZ/2[\tau, \rho] \]
where $0\not=\tau\in H^{0,1}(Spec(\bR);\bZ/2)\cong \bZ/2$ and
where 
\[\rho=-1\in \bR^*/(\bR^*)^2\cong K_1^M(\bR)/2\cong H_{et}^1(Spec(\bR);\bZ/2).\]

We recall the cycle map from the Chow ring to the etale cohomology
\[cl/2: CH^*(X)/2\to H^{2*}_{et}(X;\bZ/2).
\]
This map is also written as
$ H^{2*,*}(X;\bZ/2)\stackrel{\times \tau^*}{\to}
H^{2*,2*}(X;\bZ/2).$
We say that $x\in H_{et}^*(X;\bZ/2)$ is algebraic if
$x\in Im(cl/2)$ (for $\bZ/2$ coefficients).

Let $Q$ be an anisotropic quadric of dimension $2^n-1$
(i.e. the norm variety).  Then we have the Rost motive  $M\subset Q$ [Ro].     We will see that the etale cohomology 
$H_{et}^*(M;\bZ/2)\subset H_{et}^*(Q;\bZ/2)$
has many non-algebraic elements.  It is known [Ya2]
\[ H_{et}^{*}(M;\bZ/2)\cong \bZ/2[\rho]/(\rho^{2^{n+1}-1})
\cong \bZ/2\{1,\rho,\rho^2,...,\rho^{2^{n+1}-2}\}\]
(see also [Be] Proposition 4.13, and Theorem 1.3,5.3 or the Remark page 575 in [Ya2]).

For the restriction map to the cohomology of $\bar M=M(\bC)$, we see,
from $res(\rho)=0$, 
\[Im (res: H_{et}^*(M;\bZ/2)\to H_{et}^*(\bar M;\bZ/2))=\bZ/2.\]

The 
Chow ring is also known [Ro]
\[CH^*(M
)/2\cong \bZ/2\{1,c_0,c_1....,c_{n-1}\}, 
\quad cl(c_i)=\rho^{2^{n+1}-2^{i+1}}.\]
The cycle map $cl/2$ is injective.
The elements $c_i$ is also written as 
\[ c_i=\rho^{2^{n+1}-2^{i+1} }
\tau^{-2^n+2^{i}}  
\quad  in \ CH^*(M)/2\subset H_{et}^{2*}(M:\bZ/2)[\tau^{-1}] \]

Therefore it is immediate
\begin{lemma}
The
nonzero element $\rho^c$  in 
\[H_{et}^{c}(M;\bZ/2)\cong \bZ/2
\{\rho^c\} \quad 
with \ \  c\not =2^{n+1}-2^{i+1}\]
  is non-algebraic.
\end{lemma}
Since $M$ is a sub-motive of $Q$, we have
$H^*(M;\bZ/2)\subset H^*(Q;\bZ/2)$.
\begin{cor}
For each $1\le c\le 2^{n+1}-3$ with
$c\not =2^{n+1}-2^{i+1}$ 
the cohomology $H_{et}^{c}(Q;\bZ/2)$ (of the norm 
variety $Q$ with $deg=2^n-1$)
has a non-algebraic element.
\end{cor}

\section{quadrics with coefficients in $\bZ_2$}

In this section we consider integral coefficients case.
It is known [Ro], 
\[CH^*(M)\cong
\bZ\{1,c_0\}\oplus  \bZ/2\{c_1,...,c_{n-1}\}.\]
That is, $c_0$ is torsion free but $c_i,\ i\ge 1$ is just $2$-torsion. 

This $c_0$ is characterized as follows.  First note we can take
$M$ as an open variety (Theorem 5.11 in [Ya2]).  We consider the restriction map
\[ res :CH^*(M)\to CH^*(\bar M)\cong H^*(M(\bC))
\cong \bZ\{1,y\} \]
where $y$ is the fundamental class of 
$M(\bC)$ (and $Q(\bC))$.  It is known by Rost
$res (c_0)=2y.$  The rational coefficients case is written 
\[ CH^*(M)\otimes \bQ\cong \bQ\{1,c_0\}\cong 
\bQ\{1,y\}\cong H^{2*}(M(\bC);\bQ).\]

In this paper, the $2$-adic  integral $\bZ_2$ cohomology means the inverse limit
\[ H_{et}^*(M;\bZ_2)= 
Lim_{\infty \gets s
}H^{*,*}(M;\bZ/2^s)\]
of motivic cohomologies. 
However, we do $not$ treat $H_{et}^*(X;\bZ)$ itself, in this paper.


We recall here the Lichtenberg cohomology
[Vo1,2]
such that 
\[ H_L^{*,*'}(X;\bZ)\cong H^{*,*'}(X;\bZ)\quad for\ *\le *'+1.\]
(The right side is the motivic cohomology.)
By the five lemma, we see (for $1/s\in k$)
\[ H_L^{*,*'}(X;\bZ/s)\cong H^{*,*'}(X;\bZ/s)
\quad  for\ \ *\le *'.\]
Moreover we have
$H_L^{2*,*'}(X;\bZ/s)\cong
H^{2*}_{et}(X;\mu_{s}^{* '\otimes })$.


In this paper we consider the cycle maps to this Lichitenberg
(or motivic) cohomology in stead of  the etale cohomology itself.  The cycle map is written
\[ cl: CH^*(X)\otimes \bZ_2\cong H^{2*,*}(X;\bZ_2 )\to H_L^{2*,*}(X;\bZ_2)\cong H_{et}^{2*}(X;\bZ_2(*))\]
where $\bZ(*)$ is the Galois module, when $k=\bR$,
 it acts as $(-1)^{*}$.  
Here we write  \ \ 
\[H^{2*}_{et}(X;\bZ_2(*))=\oplus_{m\ge 0}(H^{4m}_{et}X;\bZ_2)\oplus
H^{4m+2}_{et}(X;\bZ_2(1)).  \]
Note that it is the (graded) ring.

Let $k=\bR$. Moreover let $*=even$.  Then the right hand side cohomology is written
\[ H^{2*}_{et}(X;\bZ_2(*))\cong  H^{2*}_{et}(X;\bZ_2(even))\cong
H^{2*}_{et}(X;\bZ_2(2*))\]
\[ \cong H^{2*,2*}_{L}(X;\bZ_2)\cong H^{2*,2*}(X;\bZ_2).\]
Similarly, when $*=odd$, we see
\[ H^{2*}_{et}(X;\bZ_2(*))\cong  H^{2*}_{et}(X;\bZ_2(odd))\cong
H^{2*}_{et}(X;\bZ_2(2*+1))\]
\[ \cong H^{2*,2*+1}_{L}(X;\bZ_2)\cong H^{2*,2*+1}(X;\bZ_2).\]
Thus in this paper, the cycle map means ;
\[ cl: CH^*(X)\otimes \bZ_2\to H_{et}^{2*}(X;\bZ_2(*))\cong \begin{cases}
   H^{2*,2*}(X;\bZ_2)\quad for\ *=even\\
 H^{2*,2*+1}(X;\bZ_2)\quad for\ *=odd.\end{cases}.
\]
We say that $x\in H_{et}^{2*}(X;\bZ(*)$ is non-algebraic if $x\not =0 \ mod(Im(cl))$.

The short exact sequence $0 \to \bZ \stackrel{2}{\to} \bZ\to \bZ/2\to 0$ induces the  long exact sequence of motivic cohomology
\[ ... \to H^{*-1,*}(M;\bZ/2)\stackrel{\delta}{\to} H^{*,*}(M;\bZ)
    \stackrel{2}{\to}
H^{*,*}(M;\bZ)\stackrel{r}{\to}
 H^{*,*}(M;\bZ/2) 
\to ... \] 
By Voevodsky [Vo1], [Vo2]), it is known $\beta(\tau)=\rho$
for the Bockstein operation $\beta$.
Let us write  $\delta (\tau \rho^{i-1})=\bar \rho_i\in 
H^{*,*}(M;\bZ)$ so that 
$r(\bar \rho_i)=\beta(\tau \rho^{i-1})=\rho^i$
since $r\delta=\beta$. 
Moteover  $\bar \rho_i$ is $2$-torsion from the above sequence.

Hence for all $1\le c\le 2^{n+1}-2$, we see 
$H^{c,c}(M;\bZ)\not =0$.
The same fact holds each
$H^c_{et}(M;\bZ/2^s)$ and so $H^c_{et}(M;\bZ_2)$.

\begin{lemma} Let $N=2^{n+1}-2$.  Then
\[ \bZ_2\{1,cl(c_0)\}\oplus \bZ/2\{\bar \rho_1,...,\bar \rho_{N}\}\subset 
H^*_{et}(M;\bZ_2)\oplus H^*_{et}(M;\bZ_2(1)).\] 
The element $\bar \rho_c$ with 
$c=0\ mod(4)$ and $c\not = 2^{n+1}-2^{i+1}$ 
is  a non-algebraic element.
\end{lemma}

{\bf Remark.}
When $c=2\ mod(4)$, the element
$ \bar \rho_c \in H^c(M;\bZ_2)$ but not in 
$H^c(M;\bZ_2(1)).$  So we identify here $\bar \rho_c$ is not in
$H_{et}^{2*}(M;\bZ_2(*)).$

\section{$\bZ/4$-coefficients}

At first,  we study the case $M=M_2$ with the coefficients $\bZ/4$.
The short exact sequence $0 \to \bZ/2 \to \bZ/4\to \bZ/2\to 0$ induces 
\[...\stackrel{\delta}{\to} H^{*,*'}(M;\bZ/2)
\stackrel{t}{\to}
H^{*,*'}(M;\bZ/4)\stackrel{r}{\to}
 H^{*,*'}(M;\bZ/2) \stackrel{\delta}{\to}
H^{*+1,*'}(M;\bZ/2)\to ... \]    
Here when $*\le *'$. the cohomology $H^{*,*'}(M;\bZ/2)\cong \bZ/2[\tau]\{1,\rho,...,\rho^N\}$
with $N=2^{n+1}-2$. 
The map $\delta$ is the Bockstein operation $\beta$.

For the study $H^{2*}(M_2;\bZ_2(*))$,  we only need to check $H^{2,3}(-), H^{4,4}(-),H^{6,7}(-)$.
At first we consider 
\[H^{1.3}(M;\bZ/2)\stackrel{\delta}{\to} H^{2,3}(M;\bZ/2)
\stackrel{t}{\to}
H^{2,3}(M;\bZ/4)\stackrel{r}{\to}
 H^{2,3}(M;\bZ/2) \stackrel{\delta}{\to}
H^{3,3}(M;\bZ/2). \]    
Here $H^{*,*'}(M;\bZ/2)\cong \bZ/2[\tau]\{1,\rho,...,\rho^6\}$ since $*\le *'$.
 Hence we see 
\[H^{1,3}(M;\bZ/2)\cong \bZ/2\{\rho\tau^2\},\quad 
H^{2,3}(M;\bZ/2)\cong \bZ/2\{\rho^2\tau\}.\]
Since $\delta=\beta$, we see $\delta(\tau)
=\rho$.  Hence $\delta(\rho\tau^2)=0$
and $\delta(\rho^2\tau)=\rho^3\not =0.$ 
Thus the above long exact sequence implies 
\[ 0\to H^{2,3}(M;\bZ/2)
\stackrel{t}{\to}
H^{2,3}(M;\bZ/4)\to 0. \]
So we see that
$H^{2,3}(M;\bZ/4)\cong 
\bZ/2
\{t(\rho^2
\tau)\}. $

Next we consider 
\[H^{3.4}(M;\bZ/2)\stackrel{\delta}{\to} H^{4,4}(M;\bZ/2)
\stackrel{t}{\to}
H^{4,4}(M;\bZ/4)\stackrel{r}{\to}
 H^{4,4}(M;\bZ/2) \stackrel{\delta}{\to}
H^{5,4}(M;\bZ/2). \]    
 Here we see 
\[H^{3,4}(M;\bZ/2)\cong \bZ/2\{\rho^3\tau\},\quad 
H^{4,4}(M;\bZ/2)\cong \bZ/2\{\rho^4\}.\]
 Moreover $\delta(\rho^3\tau)=\rho^4$
and $\delta(\rho^4)=0.$ 
The long exact sequence becomes
\[ 0\to H^{4,4}(M;\bZ/4)\stackrel{r}{\to}
 H^{4,4}(M;\bZ/2) \to 0.\]
Thus we see that
$ H^{4,4}(M;\bZ/4)\cong \bZ/2\{r^{-1}(\rho^4)\}.$

At last we consider 
\[H^{5,7}(M;\bZ/2)\stackrel{\delta}{\to} H^{6,7}(M;\bZ/2)
\stackrel{t}{\to}
H^{6,7}(M;\bZ/4)\stackrel{r}{\to}
 H^{6,7}(M;\bZ/2) \stackrel{\delta}{\to}
H^{7,7}(M;\bZ/2)=0. \]    
 Hence we see 
\[H^{5,7}(M;\bZ/2)\cong \bZ/2\{\rho^5\tau^2\},\quad 
H^{6,7}(M;\bZ/2)\cong \bZ/2\{\rho^6\tau\}.\]
 Here $\delta(\rho^5\tau^2)=0$
and $\delta(\rho^6\tau)=\rho^7=0.$ 
The long exact sequence becomes
\[ 0\to H^{6,7}(M;\bZ/2)
\stackrel{t}{\to}
H^{6,7}(M;\bZ/4)\stackrel{r}{\to}
 H^{6,7}(M;\bZ/2) \to 0. \] 
Thus we see that
$ grH^{6,7}(M;\bZ/4)\cong 
\bZ/2\{t(\rho^6\tau), r^{-1}(\rho^6\tau)\}.$
\begin{lemma}  We have the isomorphism
\[grH^{2*}(M;\bZ/4(*))\cong \bZ/2\{1,2\}
\oplus \bZ/2\{r^{-1}(\rho^6\tau), t(\rho^6\tau),  
 t(\rho^2\tau), r^{-1}(\rho^4)\}. \]
\end{lemma}

\section{$\bZ/2^s$-coefficients}

We generalize the above lemma to $\bZ/2^s$ coefficients.
The short exact sequence $0 \to \bZ/2^{s-1} \to \bZ/(2^s)\to \bZ/2\to 0$ induces the long exact sequence.  For example 
\[H^{1.3}(M;\bZ/2)\stackrel{\delta}{\to} H^{2,3}(M;\bZ/2^{s-1})
\stackrel{t}{\to}
H^{2,3}(M;\bZ/2^s)\stackrel{r}{\to}
 H^{2,3}(M;\bZ/2) \stackrel{\delta}{\to}
H^{3,3}(M;\bZ/2^{s-1}). \]    
Recall that 
\[H^{1,3}(M;\bZ/2)\cong \bZ/2\{\rho\tau^2\},\quad 
H^{2,3}(M;\bZ/2)\cong \bZ/2\{\rho^2\tau\}.\]
By inductive assumption,   there is $a_{s-1} \in H^{2,3}(M;\bZ/2^{s-1})$ such that $\delta(\rho\tau^2)\not= a_{s-1}$ 
(let $a_2=t(\rho^2\tau)$).
Moreover $\delta(\rho^2\tau)=\rho^3\not =0$ and so $r=0$.
Thus we see that
\[ H^{2,3}(M;\bZ/2^s)\cong \bZ/2\{t(a_{s-1})\}.\]

\begin{lemma}
In $Lim_{\infty \gets s}H^{2.3}(M,\bZ/2^s)$, the element $a_s$ does not exist.
\end{lemma}
\begin{proof}
We will show $r(a_s)\not =a_{s-1}$.
We consider the diagram
\[ \begin{CD}
@. H^{2,3}(M;\bZ/2^s) @>2>> H^{2,3}(M;\bZ/2^s) \\
@. @VVrV  @V=VV \\
  H^{1,3}(M;\bZ/2) @>{\delta}>> 
H^{2,3}(M;\bZ/2^{s-1}) @>t>> H^{2,3}(M;\bZ/2^s) 
\end{CD} \]

Since $\beta(\tau^3)=\rho\tau^2$, this element comes from $H^{1,3}(M;\bZ)$  and hence $\delta(\rho\tau^2)=0$.  So $t$ is injective.

Suppose $r(a_s)=a_{s-1}$.  By the commutativity of the above diagram, we have 
   \[ 0\not = ta_{s-1}=2a_s.\]
This is a contradiction to that $t(a_{s-1})=a_s$.
\end{proof}

Similarly we see $H^{4,4}(M;\bZ/2^s)$.

For the case $H^{6,7}(M;\bZ/2^s)$, let $b_2=t(\rho^6\tau)$.
The elements $b_s$ need not just $2$-torsion.  By the arguments as the proof of the above lemma,  we have 
\[ 0\not =t(b_{s})=2b_s \quad in\ H^{6,7}(M;\bZ/2^s).\]

\begin{thm}  For $s\ge 1$,we have the isomorphism
\[H^{2*}(M;\bZ/2^s(*))\cong \bZ/2^s\{1,r^{-1}(\rho^6\tau)\}
\oplus \bZ/2\{ a_s, r^{-1}(\rho^4)\}. \]
\end{thm}
Since $a_s$ does not exists when we take the inverse limit,
we get
\begin{cor}  The etale cohomology $H^{2*}_{et}(M;\bZ_2(*))$ is isomorphic to 
\[Lim_{\infty \gets s} H^{2*}_{et}(M;\bZ/2^s(*))\cong
 \bZ_2\{1,r^{-1}(\rho^6\tau)\}
\oplus \bZ/2\{ r^{-1}(\rho^4)\}.\]
\end{cor}

\section{general $n$}

We can consider the case $M_n$ for $n\ge 3$.
By the similar arguments, we can prove the following 
Theorem 6.1.

For example the arguments of the first parts of $\S 4$
can be changed as follows.
At first we consider them for $0\le m$ and $4m+2<2^{n+1}-2$.
We condider
\[H^{4m+1,4m+3}(M_n;\bZ/2)\stackrel{\delta}{\to}
 H^{4m+2,4m+3}(M_n;\bZ/2)
\stackrel{t}{\to}
H^{4m+2,4m+3}(M_n;\bZ/4)
\] \[\stackrel{r}{\to} 
 H^{4m+2,4m+3}(M_n;\bZ/2) 
\stackrel{\delta}{\to}
H^{4m+3,4m+3}(M_n;\bZ/2)... \]    
 Hence we see 
\[H^{4m+1,4m+3}(M_n;\bZ/2)\cong \bZ/2\{\rho^{4m+1}\tau^2\},\quad 
H^{4m+2,4m+3}(M_n;\bZ/2)\cong \bZ/2\{\rho^{4m+2}\tau\}.\]
 Thus we see that
$ H^{4m+2,4m+3}(M_n;\bZ/4)\cong \bZ/2\{a_2=t(\rho^{4m+2}\tau)\}.$
The $\bZ/2^s$-coefficients case also done similarly by lemma 5.1.

The next arguments for $H^{4m,4m}(M_n;\bZ/2^s)$ also work similarly. 

The last argument $H^{4m+2,4m+3}(M_n;\bZ/2^s)$ when $4m+2=2^{n+1}-2$ can be done similarly from the case $H^{6,7}(M;\bZ/2^s)$.

Writing $\pi=cl(c_0)$, we have the following theorem.
\begin{thm}
Let $M_n\subset Q^{2^n-1}$ be the Rost motive of the 
norm variety.  Then there are element 
$\pi\in H^{2^{n+1}-2}_{et}(M_n; \bZ_2(1))$ and 
$\bar \rho_{4m}\in
H^{4m}_{et}(M_n;\bZ_2)$ such that 
\[ H^{2*}_{et}(M_n;\bZ_2(*))\cong
\bZ_2\{1,\pi\}\oplus \bZ/2\{\bar \rho_4,
\bar \rho_8,...,\bar \rho_{2^{n+1}-4} \}\]
\[ \cong
\bZ_2\{1,\pi\}\oplus \bZ/2[\bar \rho_4]^+/(
\bar \rho_4^{2^{n-1}})
.\]
The image of the cycle map is given
\[CH^*(M_n)\otimes \bZ_2\cong
\bZ_2\{1,\pi\}\oplus \bZ/2\{\bar \rho_{2^{n+1}-2^{n}},
\bar \rho_{2^{n+1}-2^{n-1}},...,\bar \rho_{2^{n+1}-4} \} .\]
\end{thm}
\begin{cor}
The non-algebraic elements  are written
\[ H^{2*}(M_n;\bZ_2(*))/Im(cl)\cong
 \bZ/2\{\bar \rho_4,
\bar \rho_8,...,\bar \rho_{2^{n+1}-12} \}/\{\bar \rho_{2^{n+1}-2^{i+1}}|  2\le i \le n-1\} .\]
\end{cor}

\section{Pfister quadrics}

Let us write by $Q^d$ an anisotropic quadric so that
$M=M_n$ be the Rost motive of $Q^{2^n-1}$.  All anisotropic
quadrics over $\bR $  are excellent (see details $\S 3$ in [Ya1]), and its motive $M(Q^d)$ can be 
decomposed as follows.
\begin{lemma}(Rost [Ro], Lemma 3.2 in [Ya1])
  There is an isomorphism of motives
  \[M(Q^d)\cong \oplus_{i=0}^{r}M_{n_i}\otimes M(\bP^{m_i-1})
  \otimes \bT^{\otimes s_{i}}\]
where $\bT$ is the Tate motive so that $M(\bP^1)\cong \bT^0\oplus \bT$.
  \end{lemma}
Here $n_i,m_i,s_i$ are  known when $d$ is given, while they are quite 
complicated.  Here we recall only $n_i$ which is defined from  
\[ (*)\quad d+2=2^{n_0+1}-2^{n_1+1}+...+(-1)^r2^{n_r+1}\quad
for \  n_0>n_1>,,,>n_r+1\ge -1.\]

 Let $M(Q^d)$ contain $M_n\otimes \bT^{j\otimes}$. 
Then for  each 
\[2+2j\le c\le 2^{n+1}-2+2j \quad  with \ \ 
c=0\ mod(4), \ \
c\not =2^{n+1}-2^{i+1}+2j,\] 
the cohomology $H_{et}^{c}(Q^{d};\bZ_2)$ 
has a non-algebraic element.
  The cohomology $H_{et}^{2*}(Q^d;\bZ_2)$
has a non-algebraic element if and only if in the above decomposition contains $M_n\otimes \bT^{j\otimes}$
 for some $n\ge 3$.

In fact, $M_{n_0}$,  $n_0\ge 3$ is contained in $M(Q^d)$
when  $d\ge 2^3-1$ from $(*)$ above.
\begin{thm}   The  cohomology 
$H_{et}^{2*}(Q^d;Z_2(*))$ 
has  a non-algebraic element
if and only if $d\ge 2^3-1=7$.
\end{thm}

For example let $Q^d$ be the minimal or the maximal
 Pfister neighbors , i.e, $d=2^n-1$, $2^{n+1}-3$.
Then we have ($\S 6$ in[Ya1])
\[M(Q^{2^n-1})\cong M_{n}\oplus M_{n-1}\otimes M(\tilde \bP^{2^{n-1}-1}),\quad
M(Q^{2^{n+1}-3})\cong M_n\otimes  M( \bP^{2^n-2})
\]
where $M(\bP^s)\cong \bT^{0\otimes} \oplus \bT\oplus...
\oplus \bT^{s\otimes}$, \ $M(\tilde \bP^s)\cong \bT\oplus...
\oplus \bT^{s\otimes}$.

{\bf Remak.}
The decomposition for the Pfister quadrics
$Q^{2^{n+1}-2}$ is given
by \[M(Q^{2^{n+1}-2})\cong M_n\otimes M( \bP^{2^n-1}).\]  The Chow rings of odd and even quadrics are similar but different types (see (1.1) in [Ya1]).
The Chow rings of odd quadrics are simpler.
So we write down mainly these (odd dimensial) neighbors, while we have similar results for even quadrics. We note
that the minimal neighbor is the norm variety.

\begin{thm}
Let $Q^d$ be  the minimal or the maximal Pfisters neighbor
(i.e., anisotropic form of $d=2^n-1$ or $2^{n+1}-3$). Then
for each $c=0\  mod(4)$ and $0<c< 2d-8$, the cohomology $H_{et}^c(Q^d;\bZ_2)$
has a non-algebraic element.
\end{thm}
 
 \begin{proof}[Proof for the minimal Pfister neighbor]

Let $Q=Q^{2^n-1}$.
At first,  we recall for   
$c\not = 2^{n+1}-2^{i+1}=|c_{i}|$, the cohommology $H^c_{et}(Q;\bZ_2) $ contains non-algebraic element 
 from   
 the direct summand $M_n$.

We will construct non-algebraic element 
\[x\in H^{2*}_{et}(M_{n-1}\otimes \tilde \bP^{2^{n-1}-1};\bZ_2)
\quad with\ \   |x|=|c_i| \ \  for \  i> 3. \]

The following elements are in $Im(cl)$
\[
c_i' h^{2^{n-1}-2} \in H^*(M_{n-1}\otimes \bT^{2^{n-1}-2}:\bZ_2)\quad |c_i'|=2^n-2^{i+1},\ |h|=2 \] 
where $c_i'  \in CH^*(M_{n-1})$
and $h\in CH^1(\bT)$.
We consider the degree
\[ |c_i'h^{2^{n-1}-2}\bar \rho_4|=2^n-2^{i+1}+2^n-4+4=2^{n+1}-2^{i+1}=|c_i|.\]

Let $x'=\bar \rho_{2^n-2^{i+1}+4}$ such that $|x'|=|c_i'\bar \rho_4|.$
Then $x'$  is not in $Im(cl)$
when $i>3$,  from \[ |x'|=|c_i'\bar \rho_4|=2^n-2^{i+1}+4\not =2^n-2^{k+1}=|c_k'|.\]

Take $x=x'h^{2^{n-1}-2}$.  Then it is non-algebraic and 
$|x|=|c_i|.$
\end{proof}
\begin{proof}[Proof for the maximal neighbor]
We give the non-algebraic elements only when $c=0\  mod(4)$.
Let $Q=Q^{2^{n+1}-3}$.  Hence its cohomology is 
\[H^*(Q;\bZ_2)\cong H^*(M;\bZ_2)\otimes \bZ_2[h]/(h^{2^n-1}).\]

Let us write $x(4j,4m)=\bar \rho_4^jh^{2m}$.
Given $y=y(i,m)=c_ih^{2m}$, we want to find non-algebraic element
$x(4j,4m')$ such that $|y|=|x|$.

Given $y(i,m)=c_ih^{2m}$ with $i\not= 2$, take $x$ as follows
\[ x=\begin{cases} x(2^{n+1}-2^{i+1}+4, 4m-4)\quad if \ m>0\\
                         x(4,2^{n+1}-2^{i+1}-4) \quad if\ \ m=0.
\end{cases} \]
Then $x$ is non-algebraic and $|y|=|x|$.

Next let $i=2$ in $y(i,4m).$  Take 
\[ x=x(2^{n+1}-2^3-4, 4m+4)\quad for\  
4m+4\le2^{n+1}-4=|\bP^{2^n-2}| .\]
Thus we get non-algebraic element $x$ for all 
$y$ with
\[ |y|\le 2^{n+1}-2^3-4+2^{n+1}-4.\]
Since $2d=2^{n+1}-2+2^{n+1}-4$, we get the result.
\end{proof}

\section{ring structures in the cases $n=2.3$}

Let $X=Q^3=Q^{2^2-1}\supset M=M_2$.
Recall that $M(X)\cong M_2\oplus M_{1}\otimes \bT$.  Hence
the cohomology $H^{2*}_{et}(Q^3;\bZ_2(*))$ is isomorphic to the Chow ring
\[ CH^*(X)\cong \bZ_2\{1,c_0\}\oplus \bZ/2\{c_1\}
\oplus \bZ_2\{h,c_0'h\}.\]
The ring structure is also known ($\S 6$ in [Ya1]) identifying $h^3=c_0, h^2= c_0'h$
\[ \bZ_2[h]/(h^4)\oplus \bZ/2\{c_1\} \quad with \  c_1h=c_1^2=0, \ |c_1|=4.\]
\begin{lemma}  We have  ring isomorphisms 
\[ H^{2*}_{et}(Q^3;\bZ_2(*))
\cong \bZ_2[h,c_1]/(h^4, 2c_1, c_1h, c_1^2),\]
\[ H^{2*}_{et}(Q^5;\bZ_2(*))
\cong \bZ_2[h,c_1]/(h^6, 2c_1, c_1h^3, c_1^2),\]
\[ H^{2*}_{et}(Q^6;\bZ_2(*))
\cong \bZ_2[h,c_1,c_0]/(h^7, 2c_1, c_1h^4, c_1^2, hc_0-h^{4},
c_0c_1,c_0^2).\]
\end{lemma}
\begin{proof}
We see the proof for $Q^5$. From the decomposition of
the motive, we see (additively)
\[CH^*(Q^5)\otimes Z_2 \cong
(\bZ_2\{1,c_0\}\oplus \bZ/2\{c_1\})\otimes \bZ_2\{1,h,h^2\}
.\]
The result follows from  
$\bZ_2\{c_0,c_0h,c_0h^2\}\cong \bZ_2\{h^3,h^4,h^5\}.$

For the proof of the Pfister quadric $Q^6$, see $\S 6$ in [Ya1].
\end{proof}
Next we consider the case that a non-algebraic element exists.
\begin{lemma}  We have a  ring isomorphism 
\[ H^{2*}_{et}(Q^7;\bZ_2(*))
\cong 
\bZ_2[h]/(h^8)
\oplus \bZ/2[h]/(h^4)\{\bar \rho_4\}\otimes \bZ/2\{\bar \rho_4^2,\bar \rho_4^3\}\]
\[  \cong \bZ_2[h,\bar \rho_4]/(h^8, 2\bar \rho_4,  
h^4\bar \rho_4, h\bar \rho_4^2, \bar \rho_4^4)\]
where $h^7=c_0=\pi, \ c_1=\bar \rho_4^3, \ c_2=\bar \rho_4^2$
and $c_1'h=h\bar \rho_4$.
Hence we have 
\[H^{2*}_{et}(Q^7;\bZ_2(*))/(Im(cl)\cong \bZ/2\{\bar \rho_4\}.\]
\end{lemma}
\begin{proof}
From the decomposition of
the motive, we see (additively)
\[ H^{2*}_{et}(Q^7;\bZ_2(*))\cong 
H^{2*}(M_3;\bZ_2(*))\oplus H^{2*}(M_2;\bZ_2(*))\otimes \bZ_2\{h,h^2,h^3\}
.\]
Hence it can be written (with $ |c_0|=14$, $|c_1|=12$, 
$|c_2|=8,$ 
$|c_0'|=6$, $|c_1'|=4$) 
\[ \bZ_2\{1,c_0\}\oplus \bZ/2\{\bar \rho_4,c_1,c_2\}
\oplus (\bZ_2\{1,c_0'\}\oplus \bZ/2\{c_1'\})\otimes 
\bZ_2\{h,h^2,h^3\}.\]

The multiplicative structure of the Chow ring (but not the etale cohomology) is known in (Theorem 1.1) in [Ya1].
it is generated as a ring by elements $h$,$c_1,c_2$ and some element 
$u_1$ related as $c_1'$.  First note 
\[\bZ_2\{c_0'h,c_0'h^2,c_0'h^3\}\cong \bZ_2\{h^4,h^5,h^6\}. \]
Thus we have 
\[ \bZ_2\{1,h,...,h^7\}\cong \bZ_2\{1,h,h^2,h^3,hc_0',h^2c_0',h^3c_0',c_0\}.\]
So we have the above $H^{2*}(Q^7;\bZ_2(*))$ is isomorphic to
\[ \bZ_2[h]/(h^8)\oplus \bZ/2\{\bar \rho_4,c_2,c_1\}\oplus
\bZ_2\{c_1'h,c_1'h^2,c_1'h^3\}.\]
Taking $c_2=\bar \rho_4^2, \ c_1=\bar \rho_4^3,\ hc_1'=h\bar \rho_4$,
we have the result.
\end{proof}

By induction on $n$, the following corollary is easily seen.
\begin{cor}
 We have a  ring isomorphism 
\[ H^{2*}_{et}(Q^{2^n-1};\bZ_2(*))
\cong \bZ_2[h,\bar \rho_4]/(h^{2^n}, 2\bar \rho_4, h\bar \rho_4^{2^{n-2}}, \bar \rho_4h^{2^{n-1}},
 \bar \rho_4^{2^{n-1}}).\]
\end{cor}
\begin{proof}
From the decomposition of
the motive, we see (additively)
\[ H^{2*}_{et}(Q^{2^n-1};\bZ_2(*))\cong 
H^{2*}(M_n;\bZ_2(*))\oplus H^{2*}(M_{n-1};\bZ_2(*))\otimes \bZ_2\{h,...,h^{2^{n-1}-1}\}
.\]
Hence it can be written 
\[ \bZ_2\{1,c_0\}\oplus \bZ/2\{\bar \rho_4,...,,\bar \rho_{2^{n+1}-4}\} \] \[
\oplus (\bZ_2\{1,c_0'\}\oplus \bZ/2\{\bar \rho_4,...,\bar \rho _{2^n-4}\})\otimes 
\bZ_2\{h,...h^{2^{n-1}-1}\}.\]
Using  the identifying 
$\bZ_2\{c_0'h,..., c_0'h^{2^{n-1}-1}\}\cong \bZ_2\{h^{2^{n-1}},...,
h^{2^{n}-2}\}$,
the above cohomology is rewritten
\[ \bZ_2[h]/(h^{2^n})\oplus 
 \bZ/2\{\bar \rho_4,...,,\bar \rho_{2^{n+1}-4}\} 
\oplus ( \bZ/2\{\bar \rho_4,...,\bar \rho _{2^n-4}\})\otimes 
\bZ_2\{h,...h^{2^{n-1}-1}\}.\]
Relations $\bar\rho_{2^{n+1}}=0$, $\bar \rho_{2^n}h=0$,
imply $\bar \rho_4^{2^{n-1}}=0$, $\bar \rho_4^{2^{n-2}}h=0$,
and $\bar \rho_4h^{2^{n-1}}=0$.
\end{proof}

\section{non-split flag variety}

We study another (quite different) variety $X$ over $\bR$ which has the decomposition
\[M(X)\cong \sum_{,j} \lambda(j)M_2\otimes \bT^{j\otimes}.\]
Let $G=G_2$ the (split) simple exceptional group of $rank=2$,
Let $T$ be the maximal torus and $B$ is the Borel subgroup
containing $T$. The cohomology [Mi-To], [Ya3]
\[ H^*(G/T;\bZ/2)\cong 
S(t)/(b_1,b_2)\otimes \bZ/2\{1,y\} \]
\[where\quad S(t)=\bZ/2[t_1,t_2],\quad b_1=t_1^2+t_1t_2+t_2^2,\ 
b_2=t_2^3\]
with $|t_i|=2$, $|y|=6$.  The elements $t_i$ are represented by the $1$st Chern class.

We are interested the non-split version.  Let $\bG$ be a non
trivial $G$-torsor and let $X=\bG/B$ flag variety by $B$. 
Its etale cohomology is written [Ya3]
\[ grH^*_{et}(X;\bZ/2)\cong S(t)/(b_1,b_2)\otimes \bZ/2\{1,\rho,...,\rho^6\}.\]
We have relation $b_1=\rho^4$,  $b_2=\rho^6$ in 
$H^*_{et}(X;\bZ/2)$. Hence the Chow ring is written
(additively)
\[ CH^*(X)/2\cong S(t)/(b_1,b_2)\otimes \bZ/2
\{1,\rho^4,\rho^6\} \]
\[\cong S(t)/(b_1,b_2)\otimes \bZ/2
\{1, b_1,b_2\}
 \cong S(t)/(2,b_1^2,b_2^2,b_1b_2).\]

\begin{thm}  Let $X=\bG/B$  be the non-split flag variety for the simple  exceptional group $G=G_2$ over $\bR$. Then we have 
the ring isomorphism\ \ \[ H^{2*}_{et}(X;\bZ_2(*))
\cong \bZ_2\otimes S(t)/(2b_1,b_1^2,b_2^2,b_1b_2).\]
Of course, all elements are algebraic. 
\end{thm}

\end{document}